\documentclass[11pt,a4paper]{article}
\usepackage{fancyhdr}
\usepackage{amsmath}
\usepackage{amsthm}
\usepackage{amssymb}
\usepackage{wasysym}
\usepackage{paralist}
\usepackage{makeidx}
\usepackage[all]{xy}
\usepackage{mathdots}
\usepackage{yhmath}

\usepackage{young}
\usepackage[vcentermath]{youngtab}

\usepackage{graphicx}
\usepackage{epstopdf}

\input xy

\newcommand{\nc}{\newcommand}
 
 \nc{\cl}{\centerline}
 
 \nc{\SL}{{\rm SL}}
 \nc{\hatQ}{{\hat Q}}
 \nc{\sgn}{{\rm sgn}}

\nc\diag{{\rm diag}}
\renewcommand{\vert}{{\,|\,}}
\nc{\hatL}{{\hat L}}

\nc{\barE}{{\bar   E}}
\nc{\D}{{\mathcal D}}
\nc{\E}{{\mathcal E}}
\nc{\F}{{\mathcal F}}
\nc{\even}{{\rm e}}
\nc{\ep}{\epsilon}
\nc{\odd}{{\rm o}}
\nc{\Coker}{{\rm Coker}}
\nc{\olE}{{\overline E}}
\nc{\indBG}{{\rm ind}_B^G\,}
\nc{\indHG}{{\rm ind}_H^G\,}

\nc{\que}{{\mathbb Q}}
\nc{\barlambda}{{\bar\lambda}}
\nc{\barmu}{{\bar\mu}}
\nc{\barm}{{\bar m}}
\nc{\divind}{{\rm div.ind}}

\nc{\Sym}{{\rm Sym}}
\renewcommand{\dim}{{\rm dim\,}}
\newcommand{\q}{\quad}
\newcommand{\de}{\delta}

\newcommand{\bs}{\bigskip}

\renewcommand{\vert}{\,|\,}
 
\newcommand{\zed}{{\mathbb Z}}
\newcommand{\Ext}{{\rm Ext}}
\nc{\geom}{{\rm geom}}
\nc{\rep}{{\rm rep}}
\renewcommand{\inf}{{\rm inf}}
\newcommand{\Gqn}{{G_q(n)}}
\newcommand{\Bqn}{{B_q(n)}}

\newcommand{\nablaq}{{\nabla_q}}

\newcommand{\Deltaq}{{\Delta_q}}
\newcommand{\cf}{{\rm cf}}
\newcommand{\barG}{{\bar G}}
\newcommand{\barB}{{\bar B}}
\newcommand{\barT}{{\bar T}}
\newcommand{\barI}{{\bar I}}
\newcommand{\barL}{{\bar L}}
\newcommand{\barD}{{\bar D}}
\newcommand{\barr}{{\bar r}}
\newcommand{\barnabla}{{\overline  \nabla}}
\newcommand{\id}{{\rm id}}
\newcommand{\rmif}{{\rm if\q }}
\newcommand{\rmotherwise}{{\rm otherwise}}
\newcommand{\rmand}{{\rm \q and\q }}

\newcommand{\ch}{{\rm ch\,}}
\newcommand{\ind}{{\rm ind}}

\newcommand{\Hom}{{\rm Hom}}
\newcommand{\soc}{{\rm soc}}
\newcommand{\GL}{{\rm GL}}

\newtheorem{definition}{Definition}[section]
\newtheorem{proposition}[definition]{Proposition}
\newtheorem{theorem}[definition]{Theorem}
\newtheorem{lemma}[definition]{Lemma}

\newtheorem{remark}[definition]{Remark}

\begin{document}

\centerline{\bf Polynomially and Infinitesimally Injective Modules}

\bigskip

\centerline{Stephen Donkin and Haralampos Geranios}

\bigskip

{\it Department of Mathematics, University of York, York YO10 5DD}

{\tt stephen.donkin@york.ac.uk and }\\
{\tt   haralampos.geranios@epfl.ch}

\bs 

\centerline{11  May   2013}
\bs\bs\bs

\section*{Abstract}

The injective polynomial modules for a general linear group $G$ of degree $n$  are labelled by the  partitions with at most $n$ parts.  Working over an algebraically closed  field of characteristic $p$,  we consider the question of  which partitions correspond to polynomially injective modules that are also injective as modules for the restricted enveloping algebra of the  Lie algebra of $G$.  The question is  related to the  \lq\lq index of divisibility" of a polynomial module in general,  and an explicit answer is given for $n=2$.

\section{Introduction}

\bs

\q  Let $k$ be an algebraically closed field of positive characteristic.  Let $G$ be a reductive algebraic group over $k$. The indecomposable tilting modules are labelled by highest weight. It is rather easy to describe the tilting modules which are injective on restriction to the first infinitesimal subgroup $G_\inf$ in terms of highest weight, see \cite{D1}, (2.1) Proposition.  This applies  in particular if we take $G$ to be  $\GL_n(k)$, the group of invertible $n\times n$-matrices over $k$.  The injective polynomial modules correspond to the  tilting modules via Ringel duality, see \cite{R}, \cite{D1}, Section 3, so it is perhaps rather surprising that it not at all clear at the outset which polynomial injective modules are injective on restriction to $G_\inf$.    This is the problem that we study in this  paper. We shall call a rational $\GL_n(k)$-module polynomially injective if it is polynomial and injective in the category of polynomial modules. We shall call a  rational $\GL_n(k)$-module infinitesimally injective if its restriction to $G_\inf$ is injective. Thus our problem is to determine which polynomially injective modules are infinitesimally injective.

\q We believe this to be  worthy of study not only by comparison with the situation for tilting modules but because the injective polynomial modules are of interest in their own right. 
The rational cohomology of a polynomial $\GL_n(k)$-module may be calculated within the polynomial  category (i.e., using a resolution by injective modules in the  polynomial category) by a result of the first author, see e.g., \cite{SD20}, Theorem 9. This result finds an application in the work of Friedlander and Suslin, \cite{FS}, proving the finite generation of the cohomology of a finite group scheme.   The category of modules for $\GL_n(k)$ that are polynomial of degree $r$ is equivalent to the category of modules for a finite dimensional algebra, namely the Schur algebra $S(n,r)$, as in \cite{Green}.  So our problem is equivalent  to the study of the finite dimensional injective  left $S(n,r)$-modules  
or (by duality) the  finite dimensional projective right $S(n,r)$-modules that are injective on restriction to $G_\inf$.  However, we prefer to work with injective  modules since  these are closely related to the injective modules in the rational category (and in general non-zero projective modules do not exist in the rational category, \cite{SD38}).

 \q We give a solution to our problem  in terms of the \lq\lq index of divisibility" of a polynomially injective module. This leads us naturally to an investigation of the set of composition factors of the tensor product of several copies of the symmetric algebra $S(E)$ of the naturally $G$-module $E$, see Section 3.  This investigation will be taken further in a subsequent  paper, \cite{DG2}. 

\q Though our main interest is in the classical representation theory of general linear groups, the quantum case is not significantly more difficult for the problems considered here and our treatment will cover, for the most part, both cases simultaneously. The terminology of polynomially injective and infinitesimally injective modules will also be used when working with a quantised general linear group  at a root of unity.

\q The layout of the paper is as follows.  In the   preliminary Section  2, we establish the notation that we shall use in the sequel. This is rather lengthy since we need the substantial combinatorial background  and the representation theory and homological theory  associated with general linear groups 
 and to explain how this applies in the general quantised set-up. The quantisation of the general linear group that we use is that introduced by Richard Dipper and the first author in \cite{DiDo}.  In Section 3 we make a comparison between the polynomial injective modules and infinitesimal modules. We show that   injectivity of a module on restriction to the infinitesimal subgroups implies a factorisation of its character, Lemmas 3.1 and Lemma 3.2.    This leads, in the context of our problem, to the notion of the \lq\lq index of divisibility" of a polynomial module (the number of times the module is divisible by the determinant module in the polynomial category) and we make a detailed study of this notion, Proposition 3.5.  A description of the divisibility index of an injective module, in terms of the induced modules occurring in a filtration, is given in Lemma 3.7,  and  this is elaborated in terms of composition factors of symmetric powers in Lemma 3.9.  In Section 4 a precise criterion for injectivity on restriction to the infinitesimal subgroup is given in terms of the index of divisibility, Theorem 4.1. This may be seen as our main result.  We point out that in good circumstances injectivity of a polynomial module on restriction to the infinitesimal subgroup leads to a tensor product factorisation, Lemma 4.2. 

\q In the final Section 5 we use the theory established so far to give a complete analysis in the case $n=2$.  In particular we show that injectivity on restriction to the infinitesimal subgroup always gives rise to a tensor product factorisation in this case,  Proposition 5.6.

\bs\bs

\section{Preliminaries}

\subsection{Combinatorics}

\q We first set up the background notation for the representation theory of the quantum general linear groups, starting with the basic combinatorics. Let $n$ be a positive integer. We set $X(n)=\zed^n$.     We write $\Lambda(n)$ for the set of $n$-tuples of nonnegative integers. 
We write $X^+(n)$ for the set of dominant $n$-tuples of integers, i.e., the set of elements $\lambda=(\lambda_1,\ldots,\lambda_n)\in X(n)$ such that $\lambda_1\geq \cdots\geq  \lambda_n$, and write  $\Lambda^+(n)$ for $X^+(n)\bigcap \Lambda(n)$.
  We shall sometimes refer to elements of $\Lambda(n)$ as polynomial weights and elements of $\Lambda^+(n)$ as polynomial dominant weights. We define the degree $\deg(\lambda)$ of $\lambda=(\lambda_1,\ldots,\lambda_n)\in X(n)$ by $\deg(\lambda)=\lambda_1+\cdots+\lambda_n$.  For a nonnegative integer $r$ we write $\Lambda^+(n,r)$ for the set of partitions of $r$ into at most $n$ parts, i.e., the set of elements of $\Lambda^+(n)$ of degree $r$.

\q Let $e$ be a non-negative integer. We say that an element \\
 $\lambda=(\lambda_1,\ldots,\lambda_n)\in X(n)$ is column $e$-regular if $0\leq \lambda_1-\lambda_2,\ldots,\lambda_{n-1}-\lambda_n,\lambda_n<e$.  We  write $X_{\inf,e}(n)$ for the set of all column $e$-regular elements of $X(n)$. A  element $\lambda\in X(n)$ has a unique expression $\lambda=\lambda^0+e\barlambda$ with $\lambda^0\in X_{\inf,e}(n)$, $\barlambda\in X(n)$, moreover if $\lambda\in X^+(n)$ then $\barlambda\in X^+(n)$ and if  $\lambda\in\Lambda^+(n)$ then $\barlambda\in \Lambda^+(n)$.

\q For $1\leq i\leq n$ we write $\ep_i$ for the element of $X(n)$ with $i$th entry $1$ and all other entries $0$.   We write $\Phi$ for the set of all $\ep_i-\ep_j$, with $1\leq i,j\leq n$, $i\neq j$. We write $\Phi^+$ for the set of all $\ep_i-\ep_j$ with $1\leq i<j \leq n$.  We identify $X(n)$ with a subgroup of $\que^n$. We write $\rho$ for $\frac{1}{2}\sum_{1\leq i<j\leq n}(\ep_i-\ep_j)$, so that $2\rho=(n-1,n-3,\ldots,-(n-1))$.

\q  We write $\Sym(n)$ for the symmetric group on $\{1,2,\ldots,n\}$.  We write $w_0$ for the longest element of $\Sym(n)$. In the Lie theoretic context $\Phi^+$ has an interpretation as a set of positive roots so that $\rho$ is the half sum of positive roots.
However, we can usually work instead with the element  $\delta=(n-1,n-2,\ldots,1,0)\in X(n)$. The relationship between these elements is $2\rho=2\delta-(n-1)\omega$, where $\omega=(1,1,\ldots,1)$.
 We form the integral group ring $\zed X(n)$. This has $\zed$-basis of formal exponentials $\exp(\lambda)$ which multiply according to the rule $\exp(\lambda) \exp(\mu)=\exp(\lambda+\mu)$, $\lambda,\mu\in X(n)$.

 \subsection{Quantised general linear groups}

 \q Now  let $k$ a field. We use the  expression \lq\lq $G$ is a quantum group over $k$"  to indicated that we have in mind a Hopf algebra over $k$ denoted $k[G]$ and called the coordinate algebra of $G$. Suppose that $G$ and $H$ are quantum groups over $k$.  We use the expression   \lq\lq $\phi:G\to H$ is a morphism of quantum groups" to indicate that we have in mind a morphism of Hopf algebras from $k[H]$ to $k[G]$ denoted $\phi^\sharp$ and called the comorphism of $\phi$.  In this way we regard the category of quantum groups over $k$ as the opposite of the category of Hopf algebras over $k$. By the expression \lq\lq $J$ is a quantum subgroup of $G$"  we mean that $k[J]=k[G]/I$, for some Hopf ideal $I$, which we call the defining ideal of $J$.

\q  By a left (resp. right) $G$-module we mean a right (reps. left) $k[G]$-comodule.  Let $V$ be a finite dimensional left $G$-module with $k$-basis $v_i$, $i\in I$,  and structure map $\tau:V\to V\otimes k[G]$. We have the corresponding \lq\lq coefficient elements" $f_{ij}\in k[G]$ defined by the equations
 $$\tau(v_i)=\sum_{j\in I}^n  v_j\otimes f_{ji}$$
 for $i\in I$.  The $k$-span of the elements $f_{ij}$, $i,j\in I$, is called the coefficient space of $V$ and will be denote $\cf(V)$. (It is independent of the choice of basis.)  By a $G$-module we mean a left $G$-module. 
  If $\phi:G\to H$ is a morphism of quantum groups over $k$ and $V$ is an $H$-module with comodule structure map $\tau:V\to V\otimes k[H]$ we write $V^\phi$ for the space $V$ regarded as a $G$-module via the structure map $:(\id_V\otimes  \phi^\sharp)\circ \tau:V\to V\otimes k[G]$, where $\id_V$ is the identity map on $V$.

\q  Let $0\neq q\in k$. We consider the corresponding quantum general linear group $G_q(n)$, as in \cite{D2}.   Further details of the framework and proofs or precise references for the results described below may be found in \cite{D2} (see also \cite{DoStd} and \cite{Cox}).

 \q  We start with  the bialgebra $A_q(n)$. As a $k$-algebra this is defined by generators $c_{ij}$, $1\leq i,j \leq n$, subject to certain quadratic relations (see e.g., \cite{D2}, 0.22).  Comultiplication $\de:A_q(n)\to A_q(n)\otimes A_q(n)$ and the augmentation map $\ep:A_q(n)\to k$ are given by $\de(c_{ij})=\sum_{r=1}^n c_{ir}\otimes c_{rj}$ and $\ep(c_{ij})=\de_{ij}$, for $1\leq i,j\leq n$.  The algebra  $A_q(n)$ has a natural grading $A_q(n)=\bigoplus_{r=0}^\infty A_q(n,r)$ such that each  $c_{ij}$ has degree $1$. 
Each component $A_q(n,r)$ is a finite dimensional subcoalgebra  and the dual algebra is  the Schur algebra $S_q(n,r)$. 

\q The sign of a permutation $\pi\in \Sym(n)$ is denoted $\sgn(\pi)$. The quantum determinant $d_q=\sum_{\pi\in \Sym(n)} \sgn(\pi) c_{1,1\pi}\ldots c_{n,n\pi}$ is a group-like element and $A_q(n)$ has an Ore localisation $A_q(n)_{d_q}$.  The bialgebra structure of $A_q(n)$ extends uniquely to a bialgebra structure on the localisation and  this localised bialgebra is in fact a Hopf algebra. The quantum general linear group $G_q(n)$ is the quantum group whose coordinate algebra  $k[G_q(n)]=A_q(n)_{d_q}$.  This is the construction introduced in \cite{DiDo}. For the parallel development of the  case of the Manin quantisation, see \cite{PW}.

\subsection{Representation Theory}

\q We write $B_q(n)$ for the (quantum) subgroup of $G_q(n)$ whose defining ideal is generated by all $c_{ij}$ with $i<j$.   We write $T_q(n)$ for the subgroup of $G_q(n)$ whose defining ideal is generated by all $c_{ij}$ with $i\neq j$.  

\q A left $G$-module $V$ is said to be polynomial if $\cf(V)\leq A_q(n)$ and polynomial of degree $r$ if $\cf(V)\leq A_q(n,r)$.  A polynomial $\Gqn$-module $V$ has a unique module decomposition $V=\bigoplus_{r=0}^\infty V(r)$, where $V(r)$ is polynomial of degree $r$. Regarding a $\Gqn$-module polynomial of degree $r$ as an $A_q(n,r)$-comodule and hence as a module for the dual algebra $S_q(n,r)$ gives an  equivalence between the categories of  left $\Gqn$-modules polynomial of degree $r$ and left $S_q(n,r)$-modules.  A $\Gqn$-module $V$  is polynomial of degree $r$ if and only all composition factors of $V$ are polynomial of degree $r$ (see \cite{DoStd}, p263 for the quantum case and for example \cite{Jan}, II, A.3 Proposition for the classical case).  As in the classical case, we shall say that a $G$-module $V$ is polynomially injective it if it polynomial and injective in the category of polynomial $G$-modules.

\q For each $\lambda=(\lambda_1,\ldots,\lambda_n)\in X(n)$ there is a one dimensional $B_q(n)$-module $k_\lambda$ with structure map $\tau: k_\lambda\to k_\lambda\otimes k[B_q(n)]$ taking $v\in k_\lambda$ to $v\otimes (c_{11}^{\lambda_1}c_{22}^{\lambda_2}\ldots c_{nn}^{\lambda_n}+I)$, where $I$ is the defining ideal of $B_q(n)$. Moreover the modules $k_\lambda$, $\lambda\in X(n)$, form a complete set of pairwise non-isomorphic simple $B_q(n)$-modules  and the restrictions of these modules to $T_q(n)$ form a complete set of pairwise non-isomorphic simple $T_q(n)$-modules. All $T_q(n)$-modules are completely reducible. 
 For a $T_q(n)$-module $V$ and $\lambda\in X(n)$  the $\lambda$ weight space $V^\lambda$ is the sum of all submodules isomorphic to $k_\lambda$.
 We say that $\lambda\in X(n)$ is a weight of $V$ if $V^\lambda\neq 0$. 
The dimension of a finite dimensional vector space $V$ over $k$ will be denoted by $\dim V$. 
The character $\ch V$ of a finite dimensional  $T$-module $V$ is the element of $\zed X(n)$ defined by
$\ch V=\sum_{\lambda\in X(n)} \dim V^\lambda \exp(\lambda)$.

\q Recall the classification of irreducible $G_q(n)$-modules by highest weight. We have the dominance partial order on $X(n)$: for $\lambda,\mu\in X(n)$ we have $\lambda\leq \mu$ if $\mu-\lambda$ is a sum of elements of $\Phi^+$.  Then for each $\lambda\in X^+(n)$ there is an irreducible  module $L_q(\lambda)$ which has unique highest weight $\lambda$ and such that $\lambda$ occurs as a weight with multiplicity one. The modules $L_q(\lambda)$, $\lambda\in X^+(n)$, form a complete set of pairwise non-isomorphic irreducible   $\Gqn$-modules, see  Theorem 2.10(ii) and Lemma 3.2 of \cite{DoStd}.  It follows (by induction on highest weight)  that finite dimensional $\Gqn$-modules have the same composition factors (counting multiplicities) if and only if they have the same character : for the classical case, see \cite{Jan},II, 2.7 Remark.

\q The quantum determinant $d_q$ spans a one dimensional subcomodule of $k[\Gqn]$. We write $D_q$ for the $\Gqn$-module $kd_q$. Then $D_q=L_q(\omega)$.  The modules $L_q(\lambda)$, $\lambda\in \Lambda^+(n)$, form a complete set of pairwise non-isomorphic polynomial $G_q(n)$-modules, see e.g., \cite{DiDo}, 2.3.4 Theorem. For a finite dimensional  $\Gqn$-module $V$ and $\lambda\in X^+(n)$ we write $[V:L_q(\lambda)]$ for the multiplicity of $L_q(\lambda)$ as a composition factor of $V$.

\q We shall also need modules induced from $B_q(n)$ to $G_q(n)$. For $\lambda\in X(n)$ the  induced module $\ind_{B_q(n)}^{G_q(n)} k_\lambda$ is non-zero if and only if $\lambda\in X^+(n)$, see \cite{DoStd}, Lemma 3.2.  For $\lambda\in X^+(n)$ we set $\nabla_q(\lambda)=\ind_\Bqn^\Gqn k_\lambda$. Then $\nablaq(\lambda)$ is finite dimensional  (its character is the Schur symmetric function corresponding to $\lambda$) and the $\Gqn$-module socle of $\nablaq(\lambda)$ is $L_q(\lambda)$. The module $\nablaq(\lambda)$ has unique highest weight $\lambda$ and this weight occurs with multiplicity one.

\q A filtration $0=V_0\leq V_0\leq \cdots\leq V_r=V$ for a finite dimensional  $G$-module $V$ is said to be good if for each $1\leq i\leq r$ the quotient $V_i/V_{i-1}$ is either zero or isomorphic to $\nablaq(\lambda^i)$ for some $\lambda^i\in X^+(n)$.  For a finite dimensional   $\Gqn$-module $V$ admitting a good filtration,  for each $\lambda\in X^+(n)$  the multiplicity $|\{1\leq i\leq r\vert V_i/V_{i-1}\cong \nablaq(\lambda)\}|$ is independent of the choice of the good filtration (since it is determined by the character of $V$)  and will be denoted $(V:\nablaq(\lambda))$. 

\q For a finite dimensional  module $V$ for a quantum group  $G$ over $k$ we write $V^*$ for the dual space $\Hom_k(V,k)$ regarded as a $G$-module in the usual way and for  a non-negative integer $r$,  we write $V^{\otimes r}$ for the $r$-fold tensor product $V\otimes\cdots\otimes V$.  We write $D_q^{\otimes -r}$ for the dual of $D_q^{\otimes r}$. For $\lambda\in X^+(n)$ we take as a definition of the Weyl module $\Deltaq(\lambda)$ the dual module $\nablaq(-w_0\lambda)^*$. Thus $\nablaq(\lambda)$ and $\Deltaq(\lambda)$ have the same character. 

\q  For $\lambda\in \Lambda^+(n)$ we write $I_q(\lambda)$ for the injective envelope of $L_q(\lambda)$ in the category of polynomial $\Gqn$-modules.  If $\lambda$ has degree $r$ then $I_q(\lambda)$ is polynomial of degree $r$. As an $S_q(n,r)$-module $I_q(\lambda)$ is the injective envelope of $L_q(\lambda)$, in particular it is finite dimensional. 
Moreover, for $\lambda\in  \Lambda^+(n,r)$, the module $I_q(\lambda)$ has a good filtration and we have the reciprocity formula
$(I_q(\lambda):\nablaq(\mu))=[\nablaq(\mu):L_q(\lambda)]$,  \cite{DoStd}, Section 4, (6).

\q If $k$ has characteristic $0$ and $q=1$  or if  $q$ is not a root of unity then then all $\Gqn$-modules are completely reducible, \cite{DoStd}, Section 4, (8),  and so all polynomial modules are polynomially injective. Moreover,  for example by  the main result of \cite{DoRes}, if $H$ is any finite quantum subgroup of $\Gqn$ then the restriction to $H$ of an injective $\Gqn$-module is injective. In particular the trivial module is injective and hence all $H$-modules are injective. 
So the problem addressed in this paper is trivial in this case and we shall assume from now on that either $q$ is a primitive $l$th root of unity, with $l\geq 2$, or that $q=1$ and  $k$ has characteristic $p>0$. 

\q To simplify notation, we shall write  $G$ for  $\Gqn$, write $B$ for $B_q(n)$ and write $T$ for $T_q(n)$. We shall write $\nabla(\lambda)$ for $\nabla_q(\lambda)$, write $L(\lambda)$ for $L_q(\lambda)$, for $\lambda\in X(n)$, and write $D$ for $D_q$. 

 \q All this applies also when  $q=1$. In that case the coordinate functions will be denoted $x_{ij}$ and we write $\barG$ for $G_q(n)$, write $\barB$ for $B_q(n)$ and write $\barT$ for $T_q(n)$. 
 We have the Frobenius morphism $F:G\to \barG$ whose comorphism takes $x_{ij}$ to $c_{ij}^l$, in case $l>1$ and to $c_{ij}^p$ in case $l=1$. By abuse of notation we also write $F:B\to \barB$ and $F:T\to \barT$ for the restrictions of the Frobenius morphism $F:G\to \barG$.    In case $q=1$ we also write $\barnabla(\lambda)$ for $\nabla_q(\lambda)$, write $\barL(\lambda)$ for $L_q(\lambda)$, $\lambda\in  X^+(n)$,  and write $\barD$ for $ D_q$.  We also write $\barI(\lambda)$ for $I_q(\lambda)$,  $\lambda\in \Lambda^+(n)$ in that case.

\subsection{Infintesimal Theory}

\q We define  $e$ to be the smallest positive integer such that $1+q+\cdots + q^{e-1}=0$. (This  is $l$ if $q$ is a primitive $l$th root of unity, $l\geq 2$ and  is $p$ if $q=1$ and the characteristic of $k$ is $p>0$.) 
We write $X_\inf(n)$ for $X_{\inf,e}(n)$.  
 For $\lambda\in X^+(n)$ with $\lambda=\lambda^0+e\barlambda$, $\lambda^0\in X_\inf(n)$ we have,  by Steinberg's tensor product theorem, \cite{D3}, 3.2, (5), $L(\lambda)=L(\lambda^0)\otimes \barL(\barlambda)^F$. 
 
\q We write $G_\inf$ for the infinitesimal subgroup of $G$ whose defining ideal is generated by the elements $c_{ij}^e-\de_{ij}$, $1\leq i,j\leq n$.  We write $G_\inf T$ for the (quantum) Jantzen subgroup (scheme), i.e. the quantum subgroup of $G$ with defining ideal generated by the elements $c_{ij}^e$, for $1\leq i,j\leq n$, $i\neq j$. 
For each $\lambda\in X(n)$ there is an irreducible $G_\inf T$-module $\hatL_\inf (\lambda)$ which has unique highest weight $\lambda$ and such that $\lambda$ occurs as a weight of $\hatL_\inf(\lambda)$ with multiplicity $1$. The modules $\hatL_\inf(\lambda)$, $\lambda\in X(n)$, form a complete set of pairwise non-isomorphic irreducible $G_\inf T$-modules. For $\lambda\in X(n)$ the module $\hatL_\inf(e\lambda)$ is one dimensional and we denote it simply $k_{e\lambda}$.  For 
$\lambda,\mu\in X(n)$ we have $\hatL_\inf(\lambda+e\mu)\cong \hatL_\inf(\lambda)\otimes k_{e\mu}$, as $G_\inf T$-modules.  For $\lambda\in X_\inf(n)$ the restriction $L(\lambda)|_{G_\inf}$ is irreducible and the modules $L(\lambda)|_{G_\inf}$, $\lambda\in X_\inf(n)$, form a complete set of pairwise non-isomorphic irreducible $G_\inf$-modules, see \cite{D3}, Section 3.1.

\q For $\lambda\in X(n)$ we denote by $\hatQ_\inf(\lambda)$ a $G_\inf T$-module  injective envelope of $\hatL_\inf(\lambda)$.  For $\lambda,\mu\in X(n)$ we have $\hatQ_\inf(\lambda+e\mu)\cong \hatQ_\inf(\lambda)\otimes k_{e\mu}$, as $G_\inf T$-modules.  For $\lambda\in X_\inf(n)$ we write $Q_\inf(\lambda)$ for the restriction of $\hatQ_\inf(\lambda)$ to $G_\inf$. Then  $Q_\inf(\lambda)$ is the $G_\inf$-module  injective envelope of $L(\lambda)|_{G_\inf}$.  The module $\nabla((e-1)\delta)$, known as the Steinberg module,  is irreducible and injective as a $G_\inf$-module, see \cite{D3}, 3.2,(12). A $G_\inf T$-module is injective if and only if it is injective as a $G_\inf $-module.

 \q  For an element $\chi =\sum_{\xi\in X(n)}a_\xi\exp(\xi)$ of $\zed X(n)$ we write $\chi^F$ for the element $\sum_{\xi\in X(n)}a_\xi \exp(e\xi)$.  For a finite dimensional $\barT$--module $V$ we have $\ch V^F = (\ch V)^F$.

  \q Much of the above is a formal analogue of part of the representation theory of reductive groups, as in the standard text by Jantzen, \cite{Jan}.

\bs\bs

\section{Injective Modules}

\bs

\begin{lemma} Let $I$ be a finite dimensional $G_\inf T$-module which is injective as a $G_\inf $-module. Then we have

$$ \ch I =\sum_{\xi\in X_\inf(n)} \ch \hatQ_\inf (\xi)\,.\,   \ch \Hom_{G_\inf}(L(\xi),I).$$

\end{lemma}

\begin{proof} Both sides are additive on direct sums so we may assume that $I$ is indecomposable, i.e., that
$$I=\hatQ_\inf(\xi+e\tau)\cong \hatQ_\inf(\xi)\otimes k_{e\tau}$$
for some $\xi\in X_\inf(n)$, $\tau\in X(n)$, 
and the result is clear in this case.

\end{proof}

\bs

\q We are interested in $G$-modules which are injective on restriction to $G_\inf$.  We shall call such modules infinitesimally injective.

\begin{lemma}  Suppose that $\lambda^0\in X_\inf(n)$, $\barlambda\in \Lambda^+(n)$ and $I(\lambda^0+e\barlambda)$ is  infinitesimally injective. Then we have 
$$\ch I(\lambda^0+e\barlambda)= \ch \hatQ_\inf(\lambda^0)  .  \, \ch \barI(\barlambda)^F.$$
\end{lemma}

\begin{proof}  The $G_\inf$-module socle of $I(\lambda^0+e\barlambda)$ is $L(\lambda^0)\otimes \barI(\barlambda)^F$, by the argument of  \cite{DG1}, Lemma 3.2 (i),  so the result follows from Lemma 3.1.

\end{proof}

\bs

\begin{lemma} Let $\lambda\in \Lambda^+(n)$ and write $\lambda=\lambda^0+e\barlambda$ with $\lambda^0\in X_\inf(n)$, $\barlambda\in \Lambda^+(n)$.  Suppose  $I(\lambda^0+e\barlambda)$ is infinitesimally injective.  Then 
for each  $\tau\in \Lambda^+(n)$ such that $[\barnabla(\tau):\barL(\barlambda)]\neq 0$ we have 
$$\lambda_1^0+e\tau_n\geq (n-1)(e-1).$$

\end{lemma}

\begin{proof} Suppose  $[\barnabla(\tau):\barL(\barlambda)]\neq 0$. Then also  
$$(\barI(\barlambda):\barnabla(\tau))=[\barnabla(\tau):\barL(\barlambda)]\neq 0$$
 and so $\tau$ is a weight of $\barI(\barlambda)$ and hence by Lemma 3.2 if $\xi$ is any weight of $\hatQ_\inf(\lambda^0)$ then $\xi+e\tau$ is a weight of $I(\lambda)$.  Now $\hatQ_\inf(\lambda^0)$ has highest weight $w_0\lambda^0+2(e-1)\rho$, by \cite{Jan},II, 11.6 Lemma for the classical case and \cite{D3}, Section 3.2, (14), (ii),(iii) for the quantum case.   Hence $w_0\lambda^0+2(e-1)\rho+e\tau$ is a weight of $I(\lambda)$ and hence polynomial. Recalling that $2\rho=2\delta-(n-1)\omega$ we therefore have that 
$w_0\lambda^0+2(e-1)\delta-(n-1)(e-1)\omega+e\tau$ is a weight of $I(\lambda)$ and hence polynomial. In particular the final entry of this weight is non-negative, i.e.,
$$\lambda^0_1-(n-1)(e-1)+e\tau_n\geq 0,$$
in other words  $\lambda^0_1+e\tau_n\geq (n-1)(e-1)$.

\end{proof}

The notion that we now introduce will be key for our study of polynomially  and infinitesimally injective modules.

\bs

\begin{definition}\q Let $X$ be a non-zero  finite dimensional polynomial module.  We say that the divisibility  index $\divind(X)$ of  $X$ is $m$ if $m$ is the largest nonnegative integer such that $D^{\otimes -m}\otimes X$ is polynomial.  We shall say that $X$ is critical   if it has divisibility index $0$.
\end{definition}

\q Note that if a polynomial module $X$ has a weight $\lambda$ with $\lambda_n=0$ then $X$ has divisibility index $0$, for otherwise we could write $X$ as $Y\otimes D$, for some polynomial module $Y$ and a weight of $X$ would have the form $\mu+\omega$, for some weight $\mu$ of $Y$. 

\q An  arbitrary polynomial module $X$ may be written in the form $X=D^{\otimes m}\otimes Y$ when $m$ is the divisibility index of $X$ and $Y$ is critical.
We collect together some basic properties of the divisibility index.

\begin{proposition} Let $X$ be a non-zero finite dimensional polynomial module. We have:

(i) if $\divind(X)>0$ then $\divind(X)=\divind(X\otimes D^*)+1$;

(ii)  $\divind(X)=\min\{\lambda_n\vert \lambda=(\lambda_1,\ldots,\lambda_n)\in \Lambda^+(n) \hbox{ and } [X:L(\lambda)]\neq 0\}$; 

(iii)  $\divind(X)=\min\{\alpha_n\vert \alpha=(\alpha_1,\ldots,\alpha_n) \in \Lambda(n) \hbox{ and } \alpha \hbox{ is a weight of } X\}$;

(iv) if $0=X_0<X_1<X_2<\cdots<X_r=X$ is a properly ascending chain of submodules of $X$ then
$$\divind(X)=\min\{\divind(X_i/X_{i-1})\vert i=1,2,\ldots, r\};$$

(v) if  the highest weights of $X$ are $\lambda^1,\lambda^2,\ldots,\lambda^r$ then 
$$\divind(X)=\min\{\lambda^i_n\vert i=1,2,\ldots, r\};$$ 

(vi) for $\lambda\in \Lambda^+(n)$ we have 
$$\divind(L(\lambda))=\divind(\nabla(\lambda))=\divind(\Delta(\lambda))=\lambda_n;$$

(vii) if $Y$ is also a non-zero finite dimensional polynomial module then
$$\divind(X\otimes Y)=\divind(X)+\divind(Y).$$

\end{proposition}

\begin{proof} (i) Clear.

(ii) Let $r=\divind(X)$ and write $X=D^{\otimes r}\otimes Z$ with $Z$ polynomial.  Let $s=\min\{\lambda_n\vert \lambda\in \Lambda^+(n) \hbox{ and } [X:L(\lambda)]\neq 0\}$. 
Let $\lambda\in \Lambda^+(n)$ be such that  $L(\lambda)$ is a composition factor of $X$. Then $L(\lambda)\cong D^{\otimes r}\otimes L$ for some composition factor $L$ of $Z$.  Now $L\cong L(\mu)$ for some $\mu\in \Lambda^+(n)$ and $D^{\otimes r}\otimes L(\mu)\cong L(\mu+r\omega)$ so we have $\lambda=\mu+r\omega$ and so $\lambda_n=\mu_n+r\geq r$.  This proves $s\geq r$.

\q Now let  $K=X\otimes D^{\otimes -s}$. Let $\mu\in X^+(n)$ be such that $L(\mu)$ be a composition factor of $K$. Then $L(\mu)\otimes D^{\otimes s}\cong L(\mu+s\omega)$ is a composition factor of $X$ and hence $\mu+s\omega$ has last entry at least $s$ and therefore $\mu_n\geq 0$. Hence every composition factor of $K$ is polynomial and therefore $K$ is polynomial.

Hence $s\leq r$ and so $r=s$.

(iii) Let 
$$r=\min\{\lambda_n\vert \lambda=(\lambda_1,\ldots,\lambda_n)\in \Lambda^+(n) \hbox{ and } [X:L(\lambda)]\neq 0\}$$
 and 
 $$s=\min\{\alpha_n\vert \alpha=(\alpha_1,\ldots,\alpha_n) \in \Lambda(n) \hbox{ and } \alpha \hbox{ is a weight of } X\}.$$
 If $\lambda\in\Lambda^+(n)$ and $L(\lambda)$ is a composition factor of $X$ then $X^\lambda\neq 0$ and so $\lambda_n\geq s$. Since this holds for all such $\lambda$ we have $r\geq s$.  Also, if $\alpha\in \Lambda(n)$ and $X^\alpha\neq 0$ then $L^\alpha\neq 0$ for some composition factor of $X$. So there exists $\lambda\in \Lambda^+(n)$ with $[X:L(\lambda)]\neq 0$ and $L(\lambda)^\alpha\neq 0$.  Thus $\lambda_n\geq r$. But $L(\lambda)$ has unique highest weight $\lambda$ so that $\alpha\leq \lambda$ which  implies that $\alpha_n\geq \lambda_n$ and so $\alpha_n\geq r$. This holds for every weight $\alpha$ of $X$ so we have $s\geq r$. Hence $r=s$.

 (iv)  This follows from (iii).
 
 (v) Applying (iv) to a composition series reduces us to  consideration of the case in which $X=L(\lambda)$ for some $\lambda\in\Lambda^+(n)$ and the result holds in this case by (ii). 
 
 (vi) Clear from (v).
 
 (vii) Let $r=\divind(X)$, $s=\divind(Y)$. The inequality  \\
 $\divind(X\otimes Y)\geq r+s$ is clear from the definition of divisibility index.  By (iii) there is a weight $\alpha$ of $X$ with final  entry $\alpha_n=r$ and a weight $\beta$ of $Y$ with last entry $\beta_n=s$. Then $\alpha+\beta$ is a weight of $X\otimes Y$ and has last entry $\alpha_n+\beta_n=r+s$ so that $r+s\geq \divind(X\otimes Y)$ by (iii).  Hence $r+s=\divind(X\otimes Y)$.

\end{proof}

We are particularly interested in the index of divisibility of the polynomially injective modules.

\bs

\begin{remark} Let $\mu\in \Lambda^+(n)$. Note that the index of divisibility of $I(\mu)$  is at most $\deg(\mu)/n$.  Also, if $I(\mu)\cong D^{\otimes r}\otimes Z$ for some polynomial module $Z$ then  each weight $\xi$  of $I(\mu)$ has the form $\xi=r\omega+\eta$ for some $\eta\in \Lambda(n)$, in particular we have $\lambda=\mu-r\omega\in \Lambda^+(n)$.   Note also that since $Z=D^{\otimes -r}\otimes I(\mu)$ and $D$ is one dimensional, the socle of $Z$ is $D^{\otimes -r}\otimes S$, where $S$ is the socle of $I(\mu)$, i.e., $L(\mu)$.  Hence $Z$ has socle 
$$D^{\otimes -r}\otimes L(\mu)=D^{\otimes -r}\otimes L(\lambda+r\omega)=L(\lambda)$$
 and so $Z$ embeds in $I(\lambda)$. Also, for similar reasons,  $D^{\otimes r}\otimes I(\lambda)$ embeds in $I(\mu)$ and so the dimension of $I(\lambda)$ is at most the dimension of $I(\mu)$, i.e., the dimension of $Z$. It follows that the embedding of $Z$ in $I(\lambda)$ is an isomorphism. Thus if $I(\mu)=D^{\otimes r}\otimes Z$ for some polynomial module $Z$ then in fact $Z=I(\lambda)$, where $\lambda=\mu-r\omega$. 
\end{remark}

\q We now give another interpretation of this number.  Note that if $\lambda\in \Lambda^+(n)$ and $\lambda_n>0$ then we have $\nabla(\lambda)=D\otimes \nabla(\lambda-\omega)$ (clearly $D\otimes \nabla(\lambda)$ embeds in $\nabla(\lambda)$ and both have character the Schur symmetric function labelled by $\lambda$).

\bs

\begin{lemma} For $\mu\in \Lambda^+(n)$ we have 
$$\divind\, I(\mu)=\min\{\tau_n\vert \tau\in \Lambda^+(n) \hbox{ and } [\nabla(\tau):L(\mu)]\neq 0\}.$$

\end{lemma}

\begin{proof}
The module $I(\mu)$ has a good filtration in which a section $\nabla(\tau)$ appears $[\nabla(\tau):L(\mu)]$ times, for $\tau\in\Lambda^+(n)$, \cite{DoStd}, Section 4, (6).
So the result follows from Proposition  3.5 (iv) and (vi).
\end{proof}

\q We now relate the divisibility index to symmetric powers.  We have the natural left $G$-module $E$ with $k$-basis $e_1,\ldots,e_n$ and structure map $\tau:E\to E\otimes k[G]$ given by $\tau(e_i)=\sum_{j=1}^n e_j\otimes c_{ji}$, $1\leq i\leq n$. The usual symmetric algebra $S(E)$ has a natural $G$-module structure and each graded component $S^rE$ is a submodule, see \cite{DiDo},  Section 2.1.  Moreover, we have $S^rE=\nabla(r,0,\ldots,0)$, for $r\geq 0$, \cite{DoStd}, Remark 3.7. 

\q For  a finite sequence of nonnegative integers $\alpha=(\alpha_1,\alpha_2,\ldots)$  we write $S^\alpha E$ for $S^{\alpha_1}E\otimes S^{\alpha_2}E\cdots$. Note that by \cite{DoStd},  Section 4, (3) (i), each $S^\alpha E$ has a good filtration.  Let $a$ be a nonnegative integer. The tensor product $M=S(E)^{\otimes a}$ has a natural grading. For $r\geq 0$ the component in degree $r$ is $M^r=\bigoplus_{\alpha\in \Lambda(a,r)}  S^\alpha E$. 
 If $\lambda\in \Lambda^+(n)$ has degree $r$ then $(M^s:\nabla(\lambda))=0$ for $s\neq r$. We write $(S(E)^{\otimes a}:\nabla(\lambda))$ for the multiplicity $(M^r:\nabla(\lambda))$, i.e., $(S(E)^{\otimes a}:\nabla(\lambda))=\sum_{\alpha\in\Lambda(a,r)} (S^\alpha E:\nabla(\lambda))$.  Similarly we have the composition multiplicity $[S(E)^{\otimes a}:L(\lambda)]=\sum_{\alpha\in \Lambda(a,r)}[S^\alpha E:L(\lambda)]$.

\bs

\begin{lemma} Let $1\leq m\leq n$. For $\lambda\in \Lambda^+(n)$ we have $(S(E)^{\otimes m}:\nabla(\lambda))\neq 0$ if and only if $\lambda_i=0$ for all $i>m$.

\end{lemma}

\begin{proof} The condition is determined by the character of $S^\alpha E$, $\alpha\in \Lambda(m,r)$,  and follows by induction on $m$ and Pieri's formula  \cite{Mac}, I, (5.16)  for the multiplication of symmetric functions. 

\end{proof}

\begin{lemma} Let $\mu\in \Lambda^+(n)$.  Then $I(\mu)$ is critical  if and only if $L(\mu)$ is a composition factor of  $S(E)^{\otimes( n-1)}$. 
\end{lemma}

\begin{proof} Clear from Lemma 3.7 and Lemma 3.8. 
\end{proof}

\bs\bs

\section{Divisibility and Injectivity}

\q We now give a  criterion for infinitesimal injectivity.

\begin{theorem} Let $\lambda\in\Lambda^+(n)$ and write $\lambda=\lambda^0+e\barlambda$ with $\lambda^0\in X_\inf(n)$ and $\barlambda\in \Lambda^+(n)$. Then $I(\lambda)$ is infinitesimally injective if and only if 
$$\lambda^0_1+e\, \divind\, \barI(\barlambda) \geq (n-1)(e-1).$$

\end{theorem}

\begin{proof} If $I(\lambda)$ is infinitesimally injective then the inequality holds by Lemma 3.3 and Lemma 3.7.

\q So we now assume that the inequality holds. We need to show that $I(\lambda)$ is infinitesimally injective and we do this by exhibiting it as a summand of a module of the form $\nabla((e-1)\de)\otimes U$. Since the Steinberg module $\nabla((e-1)\de)$ is injective as a $G_\inf$-module, so is $\nabla((e-1)\de)\otimes U$.   In the context  
of semisimple groups such an argument is due to Humphreys and Verma and first appears in  \cite{HV}.

\q We consider the module
$$Y=W\otimes \barI(\bar\lambda)^F$$
where 
$$W=\nabla((e-1)\delta+w_0\lambda^0)\otimes \nabla((e-1)\delta)\otimes  D^{\otimes - (n-1)(e-1)}.$$
We first check that $Y$ is polynomial.
The module   has a filtration with sections 
$$Y_\theta=W\otimes \barnabla(\theta)^F$$
with $\theta\in \Lambda^+(n)$, $(\barI(\barlambda):\barnabla(\theta))\neq 0$. Hence a highest weight  of $Y$ is a highest weight of some $Y_\theta$ and so has the form
$$\mu(\theta)=2(e-1)\de+w_0\lambda^0+e\theta-(n-1)(e-1)\omega$$
for some $\theta\in \Lambda^+(n)$ with $(\barI(\barlambda):\barnabla(\theta))\neq 0$. For such $\theta$ we have $\theta_n\geq \divind\, \barI(\barlambda)$ and so 
\begin{align*}
\mu(\theta)_n&=\lambda^0_1+e\theta_n-(n-1)(e-1)\cr
&\geq \lambda^0_1+ e\, \divind\, \barI(\barlambda) - (n-1)(e-1)\cr
&\geq 0.
\end{align*}
Hence if $\nu$ is a highest weight of $Y$ then $\nu_n\geq 0$. If $L(\xi)$ is a composition factor of $Y$ then $\xi\leq \nu$ for some such $\nu$ and so $\xi_n\geq \nu_n\geq 0$.  Hence all composition factors of $Y$ are polynomial and therefore  $Y$ is polynomial.  Moreover  $Y$  has degree 
$$(e-1)\deg(\delta)+\deg(\lambda^0)+(e-1)\deg(\delta)+e\deg(\barlambda)-n(n-1)(e-1)=\deg(\lambda)$$
(since $2\deg(\delta)=n(n-1)$).

\q Our next task is to embed $I(\lambda)$ into $Y$. We begin by embedding $L(\lambda^0)$ into $W$. Consider the space 
\begin{align*}S&=\Hom_G(L(\lambda^0),W)\cr
&=\Hom_G(L(\lambda^0), \nabla((e-1)\delta+w_0\lambda^0)\otimes \nabla((e-1)\delta)\otimes D^{\otimes -(n-1)(e-1)}).
\end{align*}
We have
\begin{align*}S&=\Hom_G(L(\lambda^0)\otimes \nabla((e-1)\delta+w_0\lambda^0)^*\otimes D^{\otimes (n-1)(e-1)},\nabla((e-1)\delta))\cr
&=\Hom_G(V,\nabla((e-1)\delta))
\end{align*}
where $V=L(\lambda^0)\otimes \Delta((e-1)\delta-(n-1)(e-1)\omega-\lambda^0)\otimes D^{\otimes (n-1)(e-1)}$ and so we have
$S=\Hom_G(M,\nabla((e-1)\delta)$, where $M$ is the block component of $V$ for the block containing $L((e-1)\delta)$. Note that $V$ has unique highest weight $(e-1)\delta$, and hence $V$ is polynomial and therefore so is $M$.  But then the block $M$ has $\nabla((e-1)\delta)$ as its only composition factor (see e.g., \cite{D2}, Section 4, Theorem) and $V$, and hence $M$,  has highest weight $(e-1)\delta$ occurring with multiplicity one, hence $M=\nabla((e-1)\delta)$ and $S=k$. 
Hence  there is a unique (up to scalars)  $G$-module embedding of $L(\lambda^0)$ in  $W$.   Hence there is an embedding of $L(\lambda^0)\otimes \barI(\barlambda)^F$ in $Y$ with image $Z$, say.  By \cite{DG1}, Lemma 3.2(i) in the classical case, and by the argument of  the proof of  \cite{DG1}, Lemma 3.2(i) in general, the $G_\inf T$-module socle of $I(\lambda)$ is isomorphic to $L(\lambda^0)\otimes \barI(\barlambda)^F$ hence to $Z$.  

\q  By construction,  $Y$ is injective as a $G_\inf T$-module since it has as a  tensor factor the Steinberg module $\nabla((e-1)\delta)$, which is injective as a $G_\inf T$-module, \cite{D3}, 3.2, (12) (i).   Let $Z_0$ be a  $G_\inf T$-module injective hull of $Z$ in $Y$ containing $Z$. 
 Let $\psi$ be a $G$-module embedding of $Z$ in $I(\lambda)$.  Then $\psi$ extends to a $G$-module  homomorphism $\tilde\psi:Y\to I(\lambda)$. Moreover, we have $\soc_{G_\inf}Z=\soc_{G_\inf}Z_0$ so that $\tilde\psi$ is injective on $Z_0$.  Hence we have $\dim Z_0\leq \dim I(\lambda)$.  But $Z_0$ is the $G_\inf T$-module injective hull of $\soc_{G_\inf}I(\lambda)$ and hence also $\dim I(\lambda) \leq \dim Z_0$. Hence the restriction of $\tilde\psi$ to $Z_0$ is an isomorphism. In particular $I(\lambda)|_{G_\inf}$ is injective.

\end{proof}

\q Of course the criterion of the Proposition  applies if $\lambda^0_1\geq (n-1)(e-1)$ and we consider this special case in some detail.  We write $\lambda^0=(e-1)\delta+w_0\mu$, with $\mu \in X(n)$. Note that, writing $\lambda^0=(\lambda^0_1,\ldots,\lambda^0_n)$ and $\mu=(\mu_1,\ldots,\mu_n)$ we have
$$\mu=(\lambda^0_n,\ldots,\lambda^0_1)-(e-1)(0,1,\ldots,n-1)$$
and so  for $1\leq i<n$ we have
$$\mu_i-\mu_{i+1}=\lambda_{n+1-i}-\lambda_{n-i}+(e-1)=(e-1)-(\lambda_{n-i}-\lambda_{n-i+1})\eqno{(*)}.$$
Thus  $0\leq \mu_i-\mu_{i+1}\leq e-1$. 
Moreover, we have $\mu_n=\lambda_1^0-(n-1)(e-1)\geq 0$ so that  $\mu$ is polynomial and dominant. 
 Also $\mu_1=\lambda^0_n\leq e-1$, hence $\mu\in X_\inf(n)$.   Now by the argument of  \cite{DG1} Lemma 3.6 (ii),  we have that  the restriction of $I(\lambda^0)$ to $G_\inf T$ is $\hatQ_\inf(\lambda^0)$. 
We shall write $Q(\lambda^0)$ for $I(\lambda^0)$.  As in  \cite{DG1}, Lemma 3.2 (iii) we have $I(\lambda)=Q(\lambda^0)\otimes \barI(\barlambda)^F$.  To summarise, we have shown the following result.

\bs

\begin{lemma} Let $\lambda\in \Lambda^+(n)$. Write $\lambda=\lambda^0+e\barlambda$ with $\lambda^0\in X_\inf(n)$, $\barlambda\in \Lambda^+(n)$.  If $\lambda^0_1\geq (n-1)(e-1)$ then $I(\lambda)$ is infinitesimally injective and we have 
$$I(\lambda)\cong Q(\lambda^0)\otimes \barI(\barlambda)^F$$
where $Q(\lambda^0)=I(\lambda^0)$ and $Q(\lambda^0)|_{G_\inf}\cong \hatQ_\inf(\lambda^0)$.

\end{lemma} 

\bs

\q So far the discussion  has been limited  to  the restriction of a polynomially injective  module   to the  infinitesimal subgroup $G_\inf$.  In positive characteristic one has a tower of infinitesimal subgroups starting with $G_\inf$ and it is natural to ask when a polynomially injective module is injective on restriction to one of these subgroups. We shall see that in fact this is determined by the case considered so far. If $k$ has characteristic $p>0$ and $l=1$ (i.e., $q=1$)  we have a chain of subgroups $G_1<G_2<\cdots $ of $G$, where $G_1=G_\inf$ and more generally $G_m$ is the subgroup scheme whose defining ideal is generated by all $c_{ij}^{p^m}-\de_{ij}$, $1\leq i,j\leq n$.

\begin{proposition} Suppose that $k$ has characteristic $p>0$ and $q=1$.  Let $\lambda\in \Lambda^+(n)$ and write    $\lambda=\sum_{i\geq 0} p^i\lambda^i$ with all $\lambda^i\in X_{\inf,p}(n)$. Then $I(\lambda)|_{G_m}$ is injective if and only if $I(\sum_{j\geq r}p^{j-r}\lambda^j)|_{G_1}$ is injective for all $0\leq r<m$.

\end{proposition}

\begin{proof}   We identify $G$ with $\barG$ by identifying $k[\barG]$ with $k[G]$ via the Hopf algebra isomorphism taking  $x_{ij}$ with $c_{ij}$, $1\leq i,j\leq n$. (Then $F:G\to G$ is the classical Frobenius morphism.)

\q Note that if $I(\lambda)|_{G_m}$ is injective then $I(\lambda)|_{G_1}$ is injective since $G_1$ is a normal subgroup scheme of $G_m$.

\q We write $\lambda=\lambda^0+p\barlambda$, with $\lambda^0\in X_\inf(n)$, $\barlambda\in \Lambda^+(n)$.  Suppose that $I(\lambda)|_{G_1}$ is injective. Now $I(\lambda)|_{G_m}$ is injective if and only if $\Ext^1_{G_m}(L,I(\lambda))=0$ for every simple $G_m$-module $L$.  A  simple $G_m$-module is isomorphic to $L(\mu)\otimes L(\xi)^F$ for some $\mu\in X_\inf(n)$, $\xi\in X_{\inf,p^{m-1}}(n)$. Let $L=L(\mu)\otimes L(\xi)^F$ (with $\mu\in X_\inf(n)$, $\xi\in X_{\inf,p^{m-1}}(n)$) and $V=L^*\otimes I(\lambda)$.  Then $\Ext^1_{G_m}(L,I(\lambda))\cong H^1(G_m,L^*\otimes I(\lambda))=H^1(G_m,V)$.  \\

\q By \cite{DoStd}, Proposition 1.6 and p261, we have a Grothendieck spectral sequence with second page $H^i(G_m/G_1,H^j(G_1,V))$ converging to $H^*(G_m,V)$ (see \cite{Jan},I, 6.6 Proposition (3)  for the classical case).  In particular we have the $5$-term exact sequence 
\begin{align*}0\to &H^1(G_m/G_1,V^{G_1})\to H^1(G_m,V)\to H^1(G_1,V)^{G_m/G_1}\cr
\to &H^2(G_m/G_1,V^{G_1})\to H^2(G_m,V)
\end{align*}
and since $I(\lambda)$, and hence $V$, is injective as a $G_1$-module we  have \\
$H^1(G_m/G_1,V^{G_1})\cong H^1(G_m,V)$, i.e., 
$$H^1(G_m/G_1,\Hom_{G_1}(L(\mu)\otimes L(\xi)^F,I(\lambda))=\Ext^1_{G_m}(L(\mu)\otimes L(\xi)^F,I(\lambda)).$$
  However, the $G_1$-socle of $I(\lambda)$ is $L(\lambda^0)\otimes I(\barlambda)^F$, by \cite{DG1}, Lemma 3.2(i),  so that $\Hom_{G_1}(L(\mu)\otimes L(\xi)^F,I(\lambda))=0$ if $\mu\neq \lambda^0$. Hence we need only consider the case $\mu=\lambda^0$ and in this  case we have
\begin{align*}&\Ext^1_{G_m}(L(\mu)\otimes L(\xi)^F,I(\lambda))\cr
&\cong H^1(G_m/G_1,\Hom_{G_1}(L(\mu)\otimes L(\xi)^F,I(\lambda))\cr
&\cong H^1(G_m/G_1,\Hom_{G_1}(L(\mu)\otimes L(\xi)^F,L(\lambda^0)\otimes I(\barlambda)^F)\cr
&\cong H^1(G_m/G_1,(L(\xi)^*\otimes I(\barlambda))^F).
\end{align*}
But now it follows from \cite{Jan}, I, 9.5 Proposition, that 
$$H^1(G_m/G_1,(L(\xi)^*\otimes I(\barlambda))^F)\cong H^1(G_{m-1},L(\xi)^*\otimes I(\barlambda)).$$
This shows that $I(\lambda)$ is injective as a $G_m$-module if and only if it is injective as a $G_1$-module  and $I(\barlambda)$ is injective as a $\barG_{m-1}$-module.  Now  the result  follows by induction on $m$. 

\end{proof}

\q The next proposition gives the result in the quantum case modulo the classical case (covered in Proposition 4.3). So we assume $l>1$.  We  have a chain of subgroup schemes $G_\inf<G_1'<G_2'<\cdots$, where $G_m'$ is the quantum subgroup of $G$ with defining ideal generated by 
 the elements $c_{ij}^{lp^m}-\de_{ij}$, $1\leq i,j\leq n$.  The argument of proof is as in the classical case, but we need to convince the reader that the standard homological tools used there are available also in the quantum case. 
 
 \q A quantum group $H$ over  $k$ is finite if $k[H]$ is finite dimensional. The order of a finite quantum group $H$ over $k$ is the dimension of $k[H]$ and will be denoted $|H|$.  The images of the elements $c_{11}^{a_{11}}\ldots c_{nn}^{a_{nn}}$, $1\leq a_{11},\ldots, a_{nn} < lp^m$ under the restriction map $k[G]\to k[G_m']$ form a $k$-basis of $k[G_m']$ (cf \cite{D3}, p54) and so $|G_m'|=l^{n^2}p^{mn^2}$.  The coordinate algebra  of $k[G_m']$ as a $G_\inf$-module is isomorphic to a direct sum of copies of $k[G_\inf]$, by the main result of \cite{DoRes}.  Hence we have 
 $$\dim H^0(G_\inf,k[G_m'])=|G_m'|/|G_\inf|=p^{mn^2}=|\barG_m|.$$
 
 \q Now let $\barG_m'$ be the quantum group whose coordinate algebra is the  Hopf subalgebra of $k[G_m']$ generated by all $y_{ij}^l$, $1\leq i,j\leq n$, where $y_{ij}$ is the image of $c_{ij}$ under the restriction map $k[G]\to k[G_m']$. Then we have $H^0(G_\inf,k[G_m'])=k[\barG_m']$ by \cite{DoStd}, Proposition 1.5(ii),  or the original source by Parshall-Wang, \cite{PW}, Section 2.10.  (The hypothesis that $k[G_m']$ is 
 faithfully flat over $k[\barG_m']$ is satisfied since any  finite dimensional Hopf algebra is free over its Hopf subalgebras by the theorem of Nichols and Zoeller, \cite{NZ}.)

 \q We consider now the restriction of the Frobenius morphism written $F:G_m'\to \barG_m$, by abuse of notation. Now $F^\sharp k[\barG_m]=k[\barG_m']=H^0(G_\inf,k[G_m'])$.  In particular the restriction $\theta:k[\barG_m]\to k[\barG_m']$ is a Hopf  algebra surjection and hence, by dimensions, an isomorphism.  We write $\phi:\barG_m'\to \barG_m$ for the quantum group isomorphism with $\phi^\sharp=\theta$ and let $\psi:\barG_m\to \barG_m'$ be the inverse.  By construction the composite $F\circ \psi:\barG_m'\to \barG_m$ is the identity.

\begin{proposition}  Suppose that $k$ has characteristic $p>0$ and $q\neq 1$.  Let $\lambda\in \Lambda^+(n)$ and write    $\lambda=\lambda^0+l\barlambda$ with $\lambda^0\in X_\inf(n)$, $\barlambda\in \Lambda^+(n)$.  Then $I(\lambda)|_{G_m'}$ is injective if and only if $I(\lambda)$ is injective as a $G_\inf$-module and $\barI(\barlambda)$ is injective as a $\barG_m$-module.
\end{proposition}

\begin{proof}   Note that if $I(\lambda)|_{G_m'}$ is injective then $I(\lambda)|_{G_\inf}$ is injective by \cite{DoStd}, Proposition 1.5(iii). 

\q Suppose that $I(\lambda)|_{G_\inf}$ is injective.  Now $I(\lambda)|_{G_m'}$ is injective if and only if $\Ext^1_{G_m'}(L,I(\lambda))=0$ for each simple $G_m'$-module $L$. It follows from \cite{Cox}, Lemma 3.1. that 
$L$ is isomorphic to  $L(\mu)\otimes \barL^F$ for some $\mu\in X_\inf(n)$ and simple $\barG$-module $\barL$.  Now by \cite{DoStd}, Proposition 1.6 (i), there is a Grothendieck spectral sequence with $E_2$ page $H^i(\barG_m',H^j(G_\inf,L^*\otimes I(\lambda)))$ converging to $H^*(G_m',L^*\otimes I(\lambda))$.  Now arguing as in the proof of Proposition 4.3 we see that $I(\lambda)$ is injective, as a $G_m'$-module, if and only if $\Ext^1_{\barG_m'}(\barL^F,\barI(\barlambda)^F)=0$ for every simple $\barG_m$-module $\barL$.  

\q But now, $\psi:\barG_m'\to \barG_m$ is an isomorphism and so we have $\Ext^1_{\barG_m'}(X,Y)=\Ext^1_{\barG_m}(X^\psi,Y^\psi)$ for all $\barG_m$-modules $X,Y$.  Hence we have 
\begin{align*}\Ext^1_{\barG_m'}(\barL^F,\barI(\barlambda)^F)&=\Ext^1_{\barG_m}((\barL^F)^\psi,(\barI(\barlambda)^F)^\psi)\cr
&=\Ext^1_{\barG_m}((\barL)^{F\circ \psi},(\barI(\barlambda))^{F\circ \psi})\cr
&=\Ext^1_{\barG_m}(\barL,\barI(\barlambda)).
\end{align*}
Thus, $I(\lambda)$ is injective as a $G_m'$-module if and only if $I(\lambda)$ is injective as a $G_\inf$-module and $\barI(\barlambda)$ is injective as a $\barG_m$-module.

\end{proof}

\bs\bs

\section{The case $n=2$}

\bs

\q We take $n=2$. Let $\lambda\in \Lambda^+(2)$ and write $\lambda=\lambda^0+e\barlambda$, with $\lambda^0\in X_\inf(2)$ and $\barlambda\in \Lambda^+(2)$.    By Lemma 3.9 we have that $I(\lambda)$ is critical if and only if $L(\lambda)$ is a composition factor of a symmetric power of the natural module $E$.  We write $\barE$ for the natural module for $\barG$.
Let $r\geq 0$ and write $r=r_0+l\barr$, with $0\leq r_0<l$, $\barr\geq 0$. For $0\leq j<e$ the module $S^j(E)$ is irreducible: all non-zero weight spaces are one dimensional so it is enough to check that each monomial in $e_1,e_2,\ldots, e_n$ generates $S^j(E)$ and this follows, for example, from  \cite{Thams},  3.4 Corollary.
It  easy to check that for $r\geq 1$ we have
$$\ch S^r(E)=\begin{cases} \ch L(e-1,0). (\ch S^\barr(\bar E))^F &  \rmif r_0=e-1;\cr
  \ch L(r_0,0) . (\ch S^\barr(\barE))^F \cr+ \ch L(e-1,r_0+1). (\ch S^{\barr-1}(\barE))^F, & \rmotherwise.
\end{cases}
$$
Recall, from Section 2.3, that the composition multiplicities of a finite dimensional $G$-module are determined by its  character.  It  follows from the Steinberg tensor product theorem that $L(\lambda)$ is a composition factor of a symmetric power if and only if $\barL(\barlambda)$ is a composition factor of a symmetric power and either $\lambda^0_1=e-1$ or $\lambda^0_2=0$. In other words we have the following criterion.

\bs

\begin{lemma} Let $\lambda=\lambda^0+e\barlambda$, with $\lambda^0\in X_\inf(2)$, $\barlambda\in \Lambda^+(2)$. Then   $I(\lambda)$ is critical if and only if either: (i) $\barI(\barlambda)$ is critical and $\lambda^0_1=e-1$; or (ii) $\barI(\barlambda)$ is critical  and  $\lambda^0_2=0$.

\end{lemma}

The following explicit description follows from Lemma 5.1 and  induction.

\begin{lemma} Let $\lambda\in \Lambda^+(2)$ and write  $\lambda=\lambda^0+e\barlambda$ with $\lambda^0\in X_{\inf,e}(2)$, $\barlambda\in \Lambda^+(2)$.

(i)  Suppose $k$ has characteristic $0$. The module  $I(\lambda)$ is critical if and only if $\lambda^0_1=e-1$ or $\lambda^0_2=0$, and $\barlambda_2=0$.

(ii) Suppose    $k$ has characteristic $p>0$  and write $\barlambda=\sum_{i\geq 0} p^i\barlambda^i$, with all $\bar\lambda^i\in X_{\inf,p}(2)$.  The module $I(\lambda)$ is critical if and only if $\lambda^0_1=e-1$ or $\lambda^0_2=0$ and for all $i\geq 0$ we have $\barlambda^i_1=p-1$ or $\barlambda^i_2=0$.

\end{lemma}

\q We now have enough to explicitly describe those   $\lambda\in \Lambda^+(2)$ for which $I(\lambda)$ is infinitesimally injective. Writing $\lambda=\lambda^0+e\barlambda$ (with $\lambda^0\in X_\inf(2)$, $\barlambda\in \Lambda^+(2)$) we have by Theorem 4.1 that $I(\lambda)$ is infinitesimally injective precisely when $\lambda^0_1\geq e-1$ or $\barI(\barlambda)$ is not critical. So the following is immediate from Lemma 5.2.

\begin{proposition} Let $\lambda\in \Lambda^+(2)$ and write  $\lambda=\lambda^0+e\barlambda$ with $\lambda^0\in X_{\inf,e}(2)$, $\barlambda\in \Lambda^+(2)$.

(i) Suppose  $k$ has characteristic $0$. The module $I(\lambda)$ is infinitesimally injective if and only if $\lambda^0_1\geq e-1$ or $\barlambda_2\neq 0$.

(ii)  Suppose    $k$ has characteristic $p>0$  and write $\barlambda=\sum_{i\geq 0} p^i\barlambda^i$, with all $\bar\lambda^i\in X_{\inf,p}(2)$.  The module $I(\lambda)$ is infinitesimally injective if and only if either $\lambda^0_1\geq e-1$ or for some $i\geq 0$ we have $\barlambda^i_1\neq p-1$ and $\barlambda^i_2\neq 0$.

\end{proposition}

\q This solves our main problem in the case $n=2$ but for the sake of completeness we shall also give an explicit description of the index of divisibility. We start by using 
Lemma 5.1  to derive an expression for the divisibility index of $\lambda\in \Lambda^+(2)$ in terms of its expansion $\lambda=\lambda^0+e\barlambda$ (with $\lambda^0\in X_\inf(2)$, $\barlambda\in \Lambda^+(2)$). 

\bs

\begin{lemma} We have 
$$\divind\,I(\lambda)=\begin{cases} \lambda^0_2, \q\q  \hbox{ if  } \lambda^0_1<e -1 \hbox{ and  }  \barI(\barlambda) \hbox{ is critical}; \cr
\lambda^0_1-(e-1) +e\,\divind\,\barI(\barlambda),  \q \hbox{ otherwise.} 
\end{cases}$$

\end{lemma}

\begin{proof}  We argue by induction on the degree of $\lambda$. If $\lambda=0$ then $I(\lambda)$ is critical and $\divind\,I(\lambda)=\lambda^0_2=0$.  Now suppose that  $\lambda\in \Lambda^+(2)$ has positive degree but that the result holds for all $\mu\in\Lambda^+(2)$ of smaller degree. There are various cases to consider.

\bs

(a) $\lambda^0_1<e-1$, $\divind\,\barI(\barlambda)=0$.

\q  If $I(\lambda^0)$ is critical then  $\lambda^0_2=0$, by Lemma 5.1,  so that   $\divind\,I(\lambda)=0=\lambda^0_2$.  Otherwise $\lambda^0_2>0$ and $\lambda$ is not critical (by Lemma 5.1)  so we have
\begin{align*}\divind\,I(\lambda)&=\divind\,I(\lambda-\omega)+1=\divind\,I((\lambda^0-\omega)+e\barlambda)+1\cr
&=(\lambda^0-\omega)_2+1=\lambda^0_2-1+1=\lambda^0_2\end{align*}
using the inductive hypothesis.

\bs

(b) $\lambda^0_1<e-1$, $\divind\,\barI(\barlambda)>0$. 

\q Then $I(\lambda)$ is not critical, by Lemma 5.1,  so we get $\divind\,I(\lambda)=\divind\,I(\lambda-\omega)+1$.

\q  If $\lambda^0_2\neq 0$ then we get 
\begin{align*}\divind\,I(\lambda)&=\divind\,I(\lambda-\omega)+1=\divind\,I((\lambda^0-\omega)+e\barlambda)+1\cr
&=(\lambda^0-\omega)_1-(e-1)+e\,\divind\,\barI(\barlambda)+1\cr
&=\lambda^0_1-1-(e-1)+e\,\divind\,\barI(\barlambda)+1\cr
&=\lambda^0_1-(e-1)+e\,\divind\,\barI(\barlambda)
\end{align*}
using the inductive hypothesis.

\q  If $\lambda^0_2=0$ then again $I(\lambda)$ is not critical, by Lemma 5.1, so we get
\begin{align*}\divind\,I(\lambda)&=\divind\,I(\lambda-\omega)+1=\divind\,I(\lambda^0+(e-1)\omega+e(\barlambda-\omega))+1\cr
&=(\lambda^0+(e-1)\omega)_1-(e-1)+e\,\divind\,\barI(\barlambda-\omega)+1\cr
&=\lambda^0_1+(e-1)-(e-1)+e(\divind\,\barI(\barlambda)-1)+1\cr
&=\lambda^0_1-(e-1)+e\divind\,\barI(\barlambda)
\end{align*}
using the inductive hypothesis. 

\bs

(c) $\lambda^0_1\geq e-1$.

\q If $\lambda^0_1=e-1$ and $\barI(\barlambda)$ is critical then $I(\lambda)$ is critical, by Lemma 5.1, and $\divind\,I(\lambda)=0=\lambda^0_1-(e-1)$.   

\q  If $\lambda^0_1=e-1$, $\lambda^0_2\neq 0$ and $\barI(\barlambda)$ is not critical then $I(\lambda)$ is not critical (by Lemma 5.1) and
\begin{align*}
\divind\,I(\lambda)&=\divind\,I(\lambda-\omega)+1=\divind\,I((\lambda^0-\omega)+e\barlambda)+1\cr
&=(\lambda-\omega)^0_1-(e-1)+e\,\divind\,\barI(\barlambda)+1\cr
&=\lambda^0_1-(e-1)+e\,\divind\,\barI(\barlambda)
\end{align*}
using the inductive hypothesis.

\q If $\lambda^0=e-1$, $\lambda^0_2=0$ and $\barI(\barlambda)$ is not critical then $\lambda$ is not critical (by Lemma 5.1) and we have 
\begin{align*}
\divind\,I(\lambda)&=\divind\,I(\lambda-\omega)+1=\divind\,I((\lambda^0+(e-1)\omega)+e(\barlambda-\omega))+1\cr
&=(\lambda^0+(e-1)\omega)_1-(e-1)+e\,\divind\,\barI(\barlambda)-e+1\cr
&=\lambda^0_1-(e-1)+e\,\divind\,I(\barlambda)
\end{align*}
using the inductive hypothesis.

\q Finally if $\lambda^0_1>e-1$ then $\lambda^0_2\neq 0$ (since $\lambda^0\in X_\inf(2)$) and $I(\lambda)$ is not critical (by Lemma 5.1) so we have
\begin{align*}
\divind\,I(\lambda)&=\divind\,I(\lambda-\omega)+1=\divind\,I((\lambda^0-\omega)+e\barlambda)+1\cr
&=(\lambda^0-\omega)_1-(e-1)+e\,\divind\,\barI(\barlambda)+1\cr
&=\lambda^0_1-(e-1)+e\,\divind\,\barI(\barlambda)
\end{align*}
using the inductive hypothesis.

\q This completes the proof of the Lemma by induction on degree.

\end{proof}

We now give an explicit description of the divisibility index of injective indecomposable polynomial modules.

\bs

\begin{proposition} Let $\lambda\in \Lambda^+(2)$ and write $\lambda=\lambda^0+e\barlambda$, with $\lambda^0\in X_{\inf,e}$, $\barlambda\in \Lambda^+(2)$.

(i) Suppose $k$ has characteristic $0$. Then we have
$$\divind\, I(\lambda)=\begin{cases}\lambda^0_2, & \rmif  \lambda^0_1<e-1 \rmand \barlambda_2=0;\cr
\lambda^0_1+1+e(\barlambda_2-1), &\rmotherwise.
\end{cases}
$$

(ii) Suppose  $k$ has characteristic $p>0$ and $l=1$.  We write $\lambda=\sum_{i\geq 0} p^i\lambda^i$, with all $\lambda^i\in X_{\inf,p}(2)$.  If for all $i$ we have $\lambda^i_1=p-1$ or $\lambda^i_2=0$ 
then $\divind\, I(\lambda)=0$. Otherwise we have
$$\divind\, I(\lambda)=\begin{cases}
\alpha_1+p^m\lambda^m_2-(p^m-1), & \rmif \lambda^m_1<p-1;\cr
\alpha_1+p^m\lambda^m_1-(p^{m+1}-1), & \rmif \lambda^m_1\geq p-1
\end{cases}
$$
where $m$ is the maximal such that $\lambda^m_1\neq p-1$ or $\lambda^m_2\neq 0$ and $\alpha=\sum_{0\leq i<m} p^i\lambda^i$.

(iii)  Suppose $k$ has characteristic $p>0$  and $l>1$. We write $\barlambda=\sum_{i\geq 0} p^i\barlambda^i$, with all $\barlambda^i\in X_{\inf,p}(2)$. If $\lambda^0_1=e-1$ or $\lambda^0_2=0$ and for all $i$ we have $\barlambda^i_1=p-1$ or $\barlambda^i_2=0$ then $\divind\, I(\lambda)=0$.  If $\lambda^0_1\neq e-1$ and $\lambda^0_2\neq 0$ and for all $i$ we have $\barlambda^i_1=p-1$ or $\barlambda^i_2=0$ then 
$$\divind\, I(\lambda) =\begin{cases} \lambda^0_2, &\rmif \lambda^0_1 < e-1;\cr
\lambda^0_1-(e-1), &\rmif \lambda^0_1\geq e-1.
\end{cases}
$$
In all other cases we have
$$\divind\, I(\lambda)=\begin{cases}
\alpha_1+ep^m\barlambda^m_2-(ep^m-1), & \rmif \lambda^m_1<p-1;\cr
\alpha_1+ep^m\barlambda^m_1-(ep^{m+1}-1), & \rmif \lambda^m_1\geq p-1
\end{cases}
$$
where $m$ is maximal such that $\barlambda^m_1\neq p-1$ or $\barlambda^m_2\neq 0$ and \\
$\alpha=\lambda^0+e\sum_{0\leq i<m}p^i\barlambda$.

\end{proposition}

\begin{proof} (i) This is clear from Lemma 5.4.

(ii) For the first assertion see Lemma 5.2.  If the remaining assertion is false choose $\lambda$ for which it fails with $m$ as small as possible. Note that we must have $m>0$ by Lemma 5.4, in particular, by Lemma 5.2, $\barI(\lambda)$ is not critical.  Hence by Lemma 5.4 again we have 
$$\divind\, I(\lambda)=\lambda^0_1-(p-1)+ p\, \divind\, \barI(\barlambda).$$
By minimality of $m$ we have
$$\divind\, \barI(\barlambda)=\begin{cases}
\beta_1+p^{m-1}\lambda^m_2-(p^{m-1}-1), & \rmif \lambda^m<p-1;\cr
\beta_1+p^{m-1}\lambda^m_1-(p^m-1), & \rmif \lambda^m\geq p-1
\end{cases}
$$
where $\beta=\sum_{1\leq i<m} p^{i-1}\lambda^i$.  Hence we have
\begin{align*}\divind \, &I(\lambda)\cr
&=\begin{cases}
\lambda^0_1-(p-1)+ p(\beta_1+p^{m-1}\lambda^m_2-(p^{m-1}-1)), & \rmif \lambda^m<p-1;\cr
\lambda^0_1-(p-1)+p(\beta_1+p^{m-1}\lambda^m_1-(p^m-1)), & \rmif \lambda^m\geq p-1.
\end{cases}\cr
&=\begin{cases}(\lambda^0+p\beta)_1+p^m\lambda^m_2-(p^m-1),  & \rmif \lambda^m<p-1;\cr
(\lambda^0+p\beta)_1 + p^m\lambda^m_1 - (p^{m+1}-1), &   \rmif \lambda^m\geq p-1
\end{cases}\cr
&=\begin{cases}\alpha_1+p^m\lambda^m_2-(p^m-1),  & \rmif \lambda^m<p-1;\cr
\alpha_1 + p^m\lambda^m_1 - (p^{m+1}-1), &   \rmif \lambda^m\geq p-1.
\end{cases}
\end{align*}

(iii) This follows from (ii) and Lemma 5.4.

\end{proof}

\bs\rm

\q We now use Lemma 5.4 to show that any polynomially and infinitesimal injective indecomposable module may written in a standard form as a tensor product.  Suppose first that $I(\lambda)$ is infinitesimally injective and that $\divind\, I(\lambda)=0$.  As usual we write $\lambda=\lambda^0+e\barlambda$ with $\lambda^0\in X_\inf(2)$, $\barlambda\in \Lambda^+(2)$.  Then  either $\lambda^0_1\geq e-1$ or 
$\divind \,\barI(\barlambda)\neq 0$ by Theorem  4.1. Hence we have 
$$0=\divind\,I(\lambda)=\lambda^0_1-e+1+e\,\divind\,\barI(\barlambda).$$ 
Hence $\lambda^0_1=e-1$ and $\divind\,\barI(\barlambda)=0$.  Hence we have $I(\lambda)=Q(\lambda^0)\otimes \barI(\barlambda)^F$, where $Q(\lambda^0)=I(\lambda^0)$ with $Q(\lambda^0)|_{G_\inf T}\cong \hatQ_\inf(\lambda^0)$ as $G_\inf T$-modules, by Lemma 4.2. Now consider more  generally $\lambda\in \Lambda^+(2)$ such that  $I(\lambda)$ is infinitesimally injective.  Let $m=\divind\, I(\lambda)$. Then putting $\mu=\lambda-m\omega$ we have that $I(\mu)=D^{\otimes -m}\otimes I(\lambda)$ is infinitesimally injective and $\divind\, I(\mu)=0$. Writing $\mu=\mu^0+e\barmu$ with $\mu^0\in X_1(2)$, $\barmu\in \Lambda^+(2)$ we have
 $\mu^0_1=e-1$, $\barI(\barmu)$ critical and  $I(\mu)=Q(\mu^0)\otimes \barI(\barmu)^F$.  Hence we obtain $I(\lambda)= Q(\mu^0)\otimes D^{\otimes m}\otimes  \barI(\barmu)^F$. 

\q Now we write $m=m^0+e\barm$, with $0\leq m^0<e$, $\barm\geq 0$.  Then $\lambda=m\omega+\mu$ gives
$$\lambda=\mu^0+m^0\omega+e(\barm\omega+\barmu).$$

\q We examine the possibilities.

Case  (i): $\mu^0_2+m^0 < e$.  Then $\mu^0+m^0\omega\in X_\inf(2)$  and so we have $\lambda^0=\mu^0+m^0\omega$, $\barlambda=\barm\omega+\barmu$. 

Case (ii): $\mu^0_2+m^0\geq e$. Then  $\mu^0+m^0\omega -e\omega\in X_\inf(2)$  and we have $\lambda^0=\mu^0+m^0\omega-e\omega$ and $\barlambda=(\barm+1)\omega+\barmu$. This implies $\lambda^0_1<e-1$.

\q Note that in Case (i) we have $\lambda^0_1\geq e-1$ and in Case (ii) we have $\lambda^0_1<e-1$.  Organising the description of $I(\lambda)$ in terms of $\lambda^0$ and $\barlambda$ we therefore have the following result.

\bs

\begin{proposition} Let $\lambda\in \Lambda^+(2)$ and write $\lambda=\lambda^0+e\barlambda$ with $\lambda^0\in X_\inf(2)$, $\barlambda\in \Lambda^+(2)$. Suppose that $I(\lambda)$ is infinitesimally injective, i.e., that $\lambda^0_1\geq e-1$ or $\barI(\barlambda)$ is not critical. Let $m=\divind\,\barI(\lambda)$ and write $m=m^0+e\barm$ with $0\leq m^0<e$ and $\barm\geq 0$. 
The we have
$$
I(\lambda)=\begin{cases} Q(\lambda^0-m^0\omega)\otimes D^{\otimes m}\otimes I(\barlambda-\barm\omega)^F, & \hbox{ if }  \lambda^0_1\geq e-1\cr
Q(\lambda^0-m^0\omega+e\omega)\otimes D^{\otimes m} \otimes I(\barlambda - (\barm+1)\omega)^F,  & \hbox{ if } \lambda^0_1< e-1.
\end{cases}
$$
\end{proposition}

\bs\bs\bs


\begin{thebibliography}{99}


\bibitem{Cox}{A. G. Cox, \emph{The blocks of the $q$-Schur algebra}, Journal of Algebra, {\bf 207},  (1998), 306-325}





\bibitem{DiDo} {R.   Dipper and S.  Donkin, \emph{Quantum $\GL_n$}, Proc. Lond. Math. Soc. {\bf 63},  (1991), 165-211}

 \bibitem{SD20} {S. Donkin, \emph{Good filtrations of rational modules for reductive groups}, 
Representations of Finite Groups and Related Topics,  Proc. Symp. Pure
Math. : Am. Math. Soc. 47, volume {\bf1} (1987), 69-80.}





\bibitem{D1} {S. Donkin, \emph{On tilting modules for algebraic groups},  Math. Z. {\bf 212}, (1993), 39-60}


\bibitem{D2} {S. Donkin, \emph{On Schur Algebras and Related Algebras IV: The Blocks of the Schur Algebras}, Journal of Algebra, {\bf 168}  (1994), 400-429 }

\bibitem{SD38}{S. Donkin, \emph{On projective modules for algebraic groups}, J. Lond.
Math. Soc. {\bf54},(1996),75-88.}


\bibitem{DoStd} {Stephen Donkin, \emph{Standard Homological Properties for Quantum ${\rm GL}_n$},  Journal of Algebra {\bf 181}, (1996), 235-266}

\bibitem{DoRes}{S. Donkin, \emph{The restriction of the regular module for a quantum
group} , In : Algebraic Groups and Lie Groups -- A volume of papers ih honour
of the late R.W. Richardson", Edited bu G.I. Lehrer,  {Australian Math.
Soc. Lecture Series \bf9}, {pp. 183--188},  Cambridge University Press 1997}


\bibitem{D3} {S. Donkin, \emph{The $q$-Schur algebra},  LMS Lecture Notes 253, Cambridge University Press 1998}




\bibitem{DG1} {S. Donkin and H. Geranios, \emph{Endomorphism Algebras of Some Modules for Schur Algebras and Representation Dimension},  Algebras and Representation Theory, online publication March 2013, DOI 10.1007/s10468-013-9412-4 }

\bibitem{DG2} {S. Donkin and H. Geranios, \emph{Compostion factors of tensor products of symmetric powers},  in preparation}


\bibitem{EGS}{K. Erdmann, J. A. Green and M. Shocker , \emph{Polynomial Representations of ${\rm GL}_n$,  Second Edition with an Appendix on Schenstead Correspondence and Littelmann Paths}, Lecture Notes in Mathematics 830, Springer 2007}

\bibitem{FS}{E. Friedlander and A. Suslin,  \emph{Cohomology of finite group schemes over a field},  Invent. Math.  {\bf 127}, (1997),  235-253}

\bibitem{Green}{J. A. Green, \emph{Polynomial Representation of ${\rm GL}_n$: With an appendix on Littelmann Paths by K.Erdmann, J. A. Green and M. Schocker},  Second Edition, Lecture Notes in Mathematics 830, Springer 2006 .}

\bibitem{HV} {J. E. Humphreys and D.-N. Verma, \emph{Projective modules for finite Chevalley groups}, Bull. Amer. Math. Soc. {\bf 79}, (1973), 467-468 }


\bibitem{Jan}{J.  C.  Jantzen, \emph{Representations of Algebraic Groups}, second ed., Math. Surveys Monogr., vol 107, Amer. Math. \soc., (2003).}




\bibitem{Mac}{I. G. Macdonald, \emph{Symmetric Functions and Hall Polynomials}, 2nd Ed., Oxford Mathematical Monographs, Oxford University Press 1998.}

\bibitem{NZ}{Nichols, W. D.  and Zoeller, M. B, \emph{A Hopf algebra freeness theorem},  Amer. J. Math. {\bf 111}, (1989),  381-385.}


\bibitem{PW}{B. Parshall and Jian-pan Wang, \emph{Quantum Linear Groups}, Mem. Amer. Math. Soc. {\bf 439}, (1991).}


\bibitem{R} {C. M. Ringel, \emph{The category of modules with good filtrations over a quasi-hereditary algebra has almost split sequences}, Math. Z. {\bf208}, (1991), 209-225 }

\bibitem{Thams}{ L. Thams, \emph{The submodule structure of the quantum symmetric powers},  Bull. Austral. Math. Soc.  {\bf 50}, (1994),  29-39}











\end{thebibliography}
\end{document}